\documentclass[reqno]{amsart}

\pdfoutput=1
\usepackage[utf8]{inputenc} 
\usepackage{commath} 
 \usepackage{mathtools} 
\usepackage[colorlinks, citecolor=red, breaklinks=true]{hyperref} 
\usepackage[stretch=10, shrink=10]{microtype} 
\usepackage{cases}
\usepackage[initials,msc-links,non-compressed-cites,nobysame]{amsrefs}[2007/10/22]
\usepackage{tikz}
\usetikzlibrary{arrows.meta}

\newtheorem{theorem}{Theorem}[section]
\newtheorem{lemma}[theorem]{Lemma}

\newtheorem{corollary}[theorem]{Corollary}

\theoremstyle{definition}
\newtheorem{definition}[theorem]{Definition}

\theoremstyle{remark}
\newtheorem{remark}[theorem]{Remark}

\numberwithin{equation}{section}

\usepackage{paralist,enumitem}
\setlist{listparindent=0pt,parsep=3pt}

\newcommand{\TitleWithUrl}[1]{\IfEmptyBibField{doi}%
  {\IfEmptyBibField{url}{\textit{#1}}%
    {\IfEmptyBibField{eprint}{\href {\BibField{url}}{\textit{#1}}}{\textit{#1}}}%
    }%
  {\href {https://doi.org/\BibField{doi}}{\textit{#1}}}}
\renewcommand{\eprint}[1]{\IfEmptyBibField{url}{\url{#1}}%
  {\href {\BibField{url}}{#1}}}

\BibSpec{article}{%
    +{}  {\PrintAuthors}                {author}
    +{,} { \TitleWithUrl}               {title}
    +{.} { }                            {part}
    +{:} { \textit}                     {subtitle}
    +{,} { \PrintContributions}         {contribution}
    +{.} { \PrintPartials}              {partial}
    +{,} { }                            {journal}
    +{}  { \textbf}                     {volume}
    +{}  { \PrintDatePV}                {date}
    +{,} { \issuetext}                  {number}
    +{,} { \eprintpages}                {pages}
    +{,} { }                            {status}
    +{,} { available at \eprint}        {eprint}
    +{}  { \parenthesize}               {language}
    +{}  { \PrintTranslation}           {translation}
    +{;} { \PrintReprint}               {reprint}
    +{.} { }                            {note}
    +{.} {}                             {transition}
    +{}  {\SentenceSpace \PrintReviews} {review}
}

\BibSpec{collection.article}{%
    +{}  {\PrintAuthors}                {author}
    +{,} { \TitleWithUrl}                     {title}
    +{.} { }                            {part}
    +{:} { \textit}                     {subtitle}
    +{,} { \PrintContributions}         {contribution}
    +{,} { \PrintConference}            {conference}
    +{}  {\PrintBook}                   {book}
    +{,} { }                            {booktitle}
    +{,} { \PrintDateB}                 {date}
    +{,} { pp.~}                        {pages}
    +{,} { }                            {status}
    +{,} { available at \eprint}        {eprint}
    +{}  { \parenthesize}               {language}
    +{}  { \PrintTranslation}           {translation}
    +{;} { \PrintReprint}               {reprint}
    +{.} { }                            {note}
    +{.} {}                             {transition}
    +{}  {\SentenceSpace \PrintReviews} {review}
}
\BibSpec{book}{%
    +{}  {\PrintPrimary}                {transition}
    +{,} { \TitleWithUrl}               {title}
    +{.} { }                            {part}
    +{:} { \textit}                     {subtitle}
    +{,} { \PrintEdition}               {edition}
    +{}  { \PrintEditorsB}              {editor}
    +{,} { \PrintTranslatorsC}          {translator}
    +{,} { \PrintContributions}         {contribution}
    +{,} { }                            {series}
    +{,} { \voltext}                    {volume}
    +{,} { }                            {publisher}
    +{,} { }                            {organization}
    +{,} { }                            {address}
    +{,} { \PrintDateB}                 {date}
    +{,} { }                            {status}
    +{}  { \parenthesize}               {language}
    +{}  { \PrintTranslation}           {translation}
    +{;} { \PrintReprint}               {reprint}
    +{.} { }                            {note}
    +{.} {}                             {transition}
    +{}  {\SentenceSpace \PrintReviews} {review}
}

\BibSpec{thesis}{%
    +{}  {\PrintAuthors}                {author}
    +{.} { \TitleWithUrl}               {title}
    +{:} { \textit}                     {subtitle}
    +{,} { \PrintThesisType}            {type}
    +{,} { }                            {organization}
    +{} { \PrintDate}                  {date}
    +{,} { }                            {address}
    +{,} { \eprint}                     {eprint}
    +{,} { }                            {status}
    +{}  { \parenthesize}               {language}
    +{}  { \PrintTranslation}           {translation}
    +{;} { \PrintReprint}               {reprint}
    +{.} { }                            {note}
    +{.} {}                             {transition}
    +{}  {\SentenceSpace \PrintReviews} {review}
}

\title[Discretization principle via permutability]{Special cases of the discretization principle via permutability}

\author{Joseph Cho}
\address[Joseph Cho]{Global Leadership School, Handong Global University, Pohang, Republic of Korea}
\email{\url{jcho@handong.edu}}

\author{Mason Pember}
\address[Mason Pember]{Department of Mathematical Sciences, University of Bath, Bath, UK}
\email{\url{mason.j.w.pember@bath.edu}}

\author{Wayne Rossman}
\address[Wayne Rossman]{Department of Mathematics, Graduate School of Science, Kobe University, Kobe, Japan}
\email{\url{wayne@math.kobe-u.ac.jp}}


\newenvironment{contentsinred}{\bgroup\color{magenta}}{\egroup}
\newcommand{\by}{\begin{contentsinred}}   
\newcommand{\ey}{\end{contentsinred}}      

\begin{document}

\begin{abstract}
	We show how permutability of transforms of smooth surfaces with particular characteristics leads to discrete surfaces with discrete analogues of the same characteristics.
\end{abstract}

\maketitle

\section{Introduction}
\label{sec1}

In this brief note, we prove a result (Theorem \ref{mainthm}) that answers a query which arose during discussions at the July 2025 MATRIX meeting ``Integrable Geometry: Smooth and Discrete'' in Creswick, Australia.  

Many classes of surfaces have geometric characterizations, including special coordinate properties or curvature properties, that admit a description using integrable systems.  
The viewpoint of discrete differential geometry often hovers around the definition of integrable-systems-based discrete versions of these surfaces, reflecting the various coordinate and curvature properties.
Such discretization has been carried out in a wide variety of cases, including, for example, pseudospherical surfaces \cite{bobenko_DiscreteSurfacesConstant_1996}, constant mean curvature surfaces \cite{hertrich-jeromin_DiscreteVersionDarboux_1999}, isothermic surfaces \cite{bobenko_DiscreteIsothermicSurfaces_1996}, or $\Omega$-surfaces \cite{burstall_DiscreteOmeganetsGuichard_2023}. 
In this note, we restrict to the case of \emph{discrete isothermic surfaces} arising from the permutability of Darboux transformations \cite{bianchi_ComplementiAlleRicerche_1906} of smooth isothermic surfaces so that they are \emph{circular nets}.
(Other discretizations of isothermic surfaces include $S$-isothermic  \cite{bobenko_DiscretizationSurfacesIntegrable_1999}, $S$-conical \cite{bobenko_SconicalCMCSurfaces_2016}, or edge-constraint nets \cite{hoffmann_DiscreteParametrizedSurface_2017}, for example.)

For smooth surfaces, the integrable systems properties lead to a transformation theory for the surfaces, and the permutability of these transforms follows.
This provides grids of quadrilaterals called \emph{Bianchi quadrilaterals} that can be used to construct discrete surfaces, by taking a single point in the starting surface together with its images in the transforms. 
The principle of finding discrete analogues of various smooth surface classes in the permutability of transformations has led to the definition of many discrete surfaces.
We note here that this principle applies beyond surface theory; in fact, the principle of using transformation and permutability as a semi-discrete or discrete analogue originated in the field of integrable systems (see, for example, \cite{levi_BacklundTransformationsNonlinear_1980} or \cite{nijhoff_DirectLinearizationNonlinear_1983}).

In more detail, let $f_{0,0}:\Sigma^2 \to M$ be a map whose image is an initial surface (hence the subscript $0,0$) in 
some space form $M$ with some particular property $\mathcal{P}$. 
As we are looking at local properties of surfaces, assume $\Sigma^2$ is simply-connected.  
Suppose there is a transformation of a certain type $\mathcal{T}$ that preserves property $\mathcal{P}$, which we apply separately to produce two transformed property $\mathcal{P}$ surfaces $f_{1,0}$ and $f_{0,1}$ of $f_{0,0}$.  
Assume permutability holds, which means there is a unique fourth property $\mathcal{P}$ surface $f_{1,1}$ that is simultaneously a type $\mathcal{T}$ transform of both $f_{1,0}$ and $f_{0,1}$.
Note that all four of these surfaces are in $M$ with common domain $\Sigma^2$.  

We could repeat this process, for example, by using a type $\mathcal{T}$ transform of $f_{1,0}$ to produce a property $\mathcal{P}$ surface $f_{2,0}$, and permutability then gives us a property $\mathcal{P}$ transform $f_{2,1}$ simultaneously of both $f_{2,0}$ and $f_{1,1}$.  
Continuing in this way, we can choose $f_{m,0}$ (integers $m \geq 0$) and $f_{0,n}$ ($n \geq 0$), and all other $f_{m,n}$ ($m,n>0$) are determined by permutability.  
We then have a full grid of surfaces $f_{m,n}$ discretely parametrized by $(m,n) \in (\mathbb{Z}^+ \cup \{ 0 \})^2$.  See Figure \ref{fig:intro}.  

Now, taking one point $x$ in the common domain of the $f_{m,n}$, we have a discrete surface $\{f_{m,n}(x)\}_{m,n \geq 0}$ parametrized by $m,n$. 
We could ask if this discrete surface will have a discrete analogue of property $\mathcal{P}$.  

In many cases such as pseudospherical surfaces, isothermic surfaces, and $\Omega$-surfaces, the answer is yes.
However, the question remains unanswered in many other cases, including the classes of constant mean curvature surfaces in space forms, linear Weingarten surfaces, Guichard surfaces, and $L$-isothermic surfaces.
Each of the surface classes named here admits a transformation that keeps the special properties: Bianchi-B\"acklund transformation \cite{bianchi_LezioniDiGeometria_1903}, special cases of Lie-Darboux transformation \cite{burstall_PolynomialConservedQuantities_2019}, Eisenhart transformation \cite{eisenhart_TransformationsSurfacesGuichard_1914}, and Bianchi-Darboux transformation \cite{musso_BianchiDarbouxTransformLisothermic_2000}, respectively.

Our work will answer that indeed the permutability of these transformations gives rise to the discrete counterparts.
However, instead of checking this relationship for each surface class named here, we will use a uniform argument stemming from the integrable reduction of the case of isothermic surfaces, where the discrete counterparts are given a well-known interpretation based on the Bianchi quadrilaterals of Darboux transformations of isothermic surfaces.

The machinery that allows us to perform such integrable reduction is the following property $\mathcal{P}$: isothermic surfaces whose associated $1$-parameter family of flat connections admit certain \emph{polynomial conserved quantities} of degree $d$; otherwise said, the property $\mathcal{P}$ is that the surfaces are \emph{special isothermic surfaces of type $d$} \cite{burstall_SpecialIsothermicSurfaces_2012}.
In this case, type $\mathcal{T}$ transforms will be Darboux transforms that do not increase the degree of the polynomial conserved quantities, which we will refer to as \emph{B\"acklund transforms}.

We will show that discrete special isothermic surfaces of type $d$ \cite{burstall_DiscreteSurfacesConstant_2014, burstall_DiscreteSpecialIsothermic_2015} arise from the Bianchi quadrilaterals of B\"acklund transformations.
Interpreting the surface classes under consideration within this framework will allow us to answer positively for each case immediately.

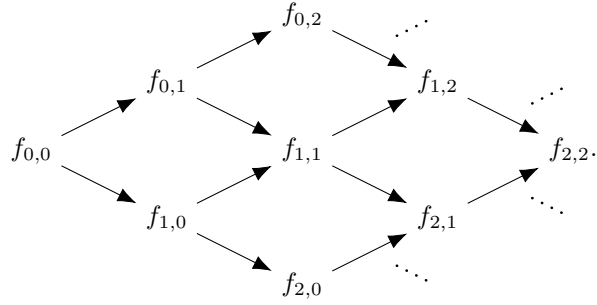
\begin{figure}
	\centering
		\begin{tikzpicture}[scale=0.9]
			\node (B) at (0,0) {$f_{0,0}$};
			\node (E) at (2,1) {$f_{0,1}$};
			\node (F) at (2,-1) {$f_{1,0}$};
			\node (H) at (4,2) {$f_{0,2}$};
			\node (I) at (4,0) {$f_{1,1}$};
			\node (J) at (4,-2) {$f_{2,0}$};
			\node (K) at (6,1) {$f_{1,2}$};
			\node (L) at (6,-1) {$f_{2,1}$};
			\node (M) at (8,0) {$f_{2,2}.$};
			\node (R) at (5,1.5) {$\phantom{f_{11}}$};
			\node (S) at (6,2) {};
			\node (T) at (7,0.5) {$\phantom{f_{11}}$};
			\node (U) at (8,1) {};
			\node (V) at (5,-1.5) {$\phantom{f_{11}}$};
			\node (W) at (6,-2) {};
			\node (X) at (7,-0.5) {$\phantom{f_{11}}$};
			\node (Y) at (8,-1) {};
			\path[-{Latex[scale=1.5]}]
				(B) edge (E)
				(B) edge (F)
				(E) edge (H)
				(E) edge (I)
				(F) edge (I)
				(F) edge (J)
				(H) edge (K)
				(I) edge (K)
				(I) edge (L)
				(J) edge (L)
				(K) edge (M)
				(L) edge (M);
			\path[line width=1pt, line cap=round, dash pattern=on 0pt off 4\pgflinewidth]
				(R) edge (S)
				(T) edge (U)
				(V) edge (W)
				(X) edge (Y);
		\end{tikzpicture}
	\caption{Grid from permutability of transforms of smooth surfaces}
	\label{fig:intro}
\end{figure}

\textbf{Acknowledgement.} We thank Fran Burstall and Udo Hertrich-Jeromin for helpful discussions, and the organizers of the program ``Integrable Geometry: Smooth and Discrete'' held at MATRIX for the opportunity to clarify the point made here. We gratefully acknowledge the partial support from: Japan Society for the Promotion of Science (Grant-in-Aid for Scientific Research (C) 23K03091) and NRF of Korea funded by MSIT (Korea-Austria Scientific and Technological Cooperation RS-2025-1435299).

\section{Permutability as a bridge between smooth and discrete surfaces}

\subsection{Smooth surfaces}\label{sec:smooth}

For non-negative integers $p$ and $q$, let us take $(p+q+2)$-dimensional pseudo-Euclidean space $\mathbb{R}^{p+1,q+1}$ with 
(pseudo-)metric denoted by $( \cdot , \cdot )$, of signature $(p+1,q+1)$, containing a $(p+q+1)$-dimensional light cone $\mathcal{L}$.
A $(p+q)$-dimensional (constant sectional curvature) space form $M$ with signature $(p,q)$ is then given by
	\[ M=\{ X \in \mathcal{L} : ( X,\mathfrak{q} ) = -1 \} \]
for some nonzero vector $\mathfrak{q} \in \mathbb{R}^{p+1,q+1}$ called the \emph{space form vector}.  
Only the direction of $X \in \mathcal{L}$ is needed for it to determine a point in $M$, and so it is convenient to projectivize $\mathcal{L}$ to $\mathbb{P}(\mathcal{L}) \subset \mathbb{P}(\mathbb{R}^{p+1,q+1})$, giving us the conformal $(p,q)$-sphere.
Consider a surface in the conformal $(p,q)$-sphere
	\[
		f :\Sigma^2 \to \mathbb{P}(\mathcal{L})
	\] 
which we also can view as a null line bundle in the trivial bundle $\Sigma^2 \times  \mathbb{R}^{p+1,q+1}$.

A surface $\mathfrak{f}$ in the space form $M$ determined by the space form vector $\mathfrak{q}$ is recovered locally from a surface $f$ in the conformal $(p,q)$-sphere via
	\[
		\mathfrak{f} = f \cap M,
	\]
wherever $f$ is not orthogonal to $\mathfrak{q}$.

\subsubsection{Isothermic surfaces and Darboux transformations}
A surface $f:\Sigma^2 \to \mathbb{P}(\mathcal{L})$ is called isothermic when it admits conformal curvature line coordinates around every non-umbilical point.
As we are dealing only with local theory of isothermic surfaces, we will assume throughout that $f$ is umbilic-free on $\Sigma^2$, and admits conformal curvature line coordinates on its domain.

Then a surface is isothermic exactly when there exists a $\wedge^2 \mathbb{R}^{p+1,q+1}$-valued non-zero closed $\eta \in \Omega^1(f \wedge f^\perp)$ defined on the tangent bundle (see \cite{burstall_ConformalSubmanifoldGeometry_2010, burstall_SpecialIsothermicSurfaces_2012}).
Here, $f^\perp$ refers to the subbundle perpendicular to the null line bundle given by $f$, and we are identifying
	\[
		\wedge^2 \mathbb{R}^{p+1,q+1} \cong \mathfrak{o}(p+1,q+1).
	\]
Locally, the closedness of the $1$-form $\eta$ is equivalent to the flatness of the (metric) connection
	\[
		\Gamma(t) = \dif{} + t \eta
	\]
for any $t \in \mathbb{R}$ defined on the trivial bundle.
For this reason, we call $\{\Gamma(t) : t \in \mathbb{R}\}$ the associated family of flat connections of $f$.

Under these settings, Darboux transformations of isothermic surfaces can be given as follows:
\begin{definition}[{\cite[Definition~5.4.8]{hertrich-jeromin_IntroductionMobiusDifferential_2003}}]
	Let $f, \hat{f}: \Sigma^2 \to \mathbb{P}(\mathcal{L})$ be a pair of isothermic surfaces into the conformal $(p,q)$-sphere, and let $\Gamma(t)$ denote the associated family of flat connections of $f$.
	If $\hat{f}: \Sigma^2 \to \mathbb{P}(\mathcal{L})$ is $\Gamma(\mu)$-parallel, that is,
		\[
			\Gamma(\mu) \hat f \parallel \hat f,
		\]
	then $f$ and $\hat{f}$ are called a \emph{Darboux pair} with respect to parameter $\mu$.
	In this case, one surface is called a \emph{Darboux transform} of the other surface with respect to the parameter $\mu$.
\end{definition}
Then the respective associated families of flat connections are related by a gauge transformation, denoted by $\bullet$,
	\[
		\dif{} + t \hat \eta = \hat \Gamma(t) = \Gamma_f^{\hat f}(1-\tfrac{t}{\mu}) \bullet \Gamma(t)  := \Gamma_f^{\hat f}(1-\tfrac{t}{\mu}) \circ \Gamma(t) \circ \left(\Gamma_f^{\hat f}(1-\tfrac{t}{\mu})\right)^{-1}
	\]
where $\Gamma_{f}^{\hat f}(1-\tfrac{t}{\mu})$ is the Lorentz boost \cite{burstall_IsothermicSubmanifoldsSymmetric_2011}
	\[
		\Gamma_{f}^{\hat f}(1-\tfrac{t}{\mu}) Y = 
			\begin{cases}
				(1-\frac{t}{\mu}) Y &\text{for $Y \in \hat f$,}\\
				(1-\frac{t}{\mu})^{-1} Y &\text{for $Y \in f$,}\\
				Y &\text{for $Y \perp \hat f,f$.}
			\end{cases}
	\]
Generally, $\hat{f}$ and $f$ will not be orthogonal; however, for the indefinite case of $q > 0$, one can choose the spectral parameter $\mu = \infty$ to obtain an \emph{isotropic Darboux transformation}:
\begin{definition}[{cf.\ \cite[Definition~4.4]{burstall_DiscreteOmeganetsGuichard_2023}, \cite[Lemma~4.5.1]{clarkeIntegrabilitySubmanifoldGeometry2012}}]
	Let $f,\hat{f} : \Sigma^2 \to \mathbb{P}(\mathcal{L})$ be a pair of isothermic surfaces with respective associated closed $1$-forms $\eta$ and $\hat{\eta}$.
	Then they are called an \emph{isotropic Darboux pair} if there exists some $\tau \in \Gamma(f \wedge \hat{f})$ such that
		\[
			\hat{\eta} = \eta + \dif{\tau}.
		\]
	Otherwise said, denoting the respective associated $1$-parameter families of flat connections of $f$ and $\hat{f}$ as $\Gamma(t)$ and $\hat{\Gamma}(t)$, we have
		\[
			\hat{\Gamma}(t) = \exp(t \tau) \bullet \Gamma(t).
		\]
\end{definition}

Referring to two successive Darboux transformations as two-step Darboux transformations, we know that such two-step Darboux transformations with finite spectral parameters permute.
Likewise, isotropic Darboux transformations and Darboux transformations with finite spectral parameter also permute.
\begin{theorem}[{cf.\ \cite[Proposition~4.16]{burstall_DiscreteOmeganetsGuichard_2023}, \cite[Proposition~4.5.5]{clarkeIntegrabilitySubmanifoldGeometry2012}}]\label{thm:permiso}
	Suppose that $f^+$ is an isothermic surface.
	If $\hat{f}^+$ is a Darboux transformation of $f^+$ with finite spectral parameter $\mu$, while $f^-$ is an isotropic Darboux transformation of $f^+$, then there uniquely exists a fourth surface $\hat{f}^-$ that is simultaneously a Darboux transformation of $f^-$ with finite spectral parameter $\mu$, and an isotropic Darboux transformation of $\hat{f}^+$.
\end{theorem}

\subsubsection{Integrable reduction of isothermicity via polynomial conserved quantities}

We recall the definition of polynomial conserved quantities.
\begin{definition}[{\cite[Definition~2.1, Definition~2.3]{burstall_SpecialIsothermicSurfaces_2012}}]
For an isothermic surface $f : \Sigma^2 \to \mathbb{P}(\mathcal{L})$ with associated family of flat connections $\Gamma(t)$, let $p(t)$ be a parallel section of $\Gamma(t) $ that is polynomial in $t$ with coefficients $p^{(\iota)} : \Sigma^2 \to \mathbb{R}^{p+1,q+1}$, that is,
	\[
		p(t) = p^{(0)} + p^{(1)}t + p^{(2)}t^2 + \cdots + p^{(d)}t^d,
	\]
for some $d \in \mathbb{N}$, satisfying
	\[
		\Gamma(t) p(t) = 0.
	\]
Then $p(t)$ is called a \emph{polynomial conserved quantity}\footnote{The concept of polynomial conserved quantities originates from the notion of polynomial Killing fields introduced in \cite{burstall_HarmonicToriSymmetric_1993}.}  of $f$.
When an isothermic surface admits a polynomial conserved quantity of degree $d$, we say that the surface is a \emph{special isothermic surface of type $d$}.
\end{definition}

\begin{remark}
	When the degree $d$ of a polynomial conserved quantity $p(t)$ is zero, then $p(t)$ is called a constant conserved quantity.
	In fact, if $p(t)$ is a constant conserved quantity, then $p(t) = p^{(0)}$ is a constant vector.
	When $d = 1$, then $p(t)$ is called a linear conserved quantity.
\end{remark}
	
Thus, if $f$ is a special isothermic surface of type $d$ with polynomial conserved quantity $p(t)$, then 
	\begin{equation}\label{eq:pphat}
		\hat p(t) = \Gamma_f^{\hat f}(1-\tfrac{t}{\mu}) p(t)
	\end{equation} 
will be a conserved quantity of $\hat f$, which may not necessarily be polynomial in $t$.
In the case that $p(\mu) \perp \hat f$ at one point of $\Sigma^2$, then this is so at every point of $\Sigma^2$, since $p(\mu)$ and $\hat f$ are both parallel for $\Gamma(\mu)$, and $\hat p(t)$ becomes polynomial of degree at most $d$.
Conversely, if $\hat p(t)$ is a polynomial in $t$, then $p(\mu) \perp \hat f$ and $\text{degree}(\hat p(t)) \leq d$.
\begin{definition}[{\cite[Theorem~3.1]{burstall_SpecialIsothermicSurfaces_2012}}]
	Let $f, \hat{f}: \Sigma^2 \to \mathbb{P}(\mathcal{L})$ be a Darboux pair with respect to parameter $\mu$, and suppose $f$ is a special isothermic surface of type $d$ with polynomial conserved quantity $p(t)$.
	If $p(\mu) \perp \hat f$, then $\hat{f}$ is also a special isothermic surface of type $d$, referred to as a \emph{B\"acklund transform} of $f$ with respect to parameter $\mu$.
\end{definition}

\subsection{Discrete surfaces}
Let $D^2 \subset  \mathbb{Z}^2$ be simply connected in the sense of \cite[Section~2.3]{burstall_DiscreteOmeganetsGuichard_2023} with $(m,n) \in D^2$.
In our setting, it is sufficient to consider the collection of the vertices of unit squares where those squares form a simply connected set in $\mathbb{R}^2$.
We will often use $i = (m,n)$, $j=(m+1,n)$, $k=(m+1,n+1)$ and $\ell=(m,n+1)$ to denote the vertices around an elementary quadrilateral $(ijk\ell)$.

Consider a discrete map $F: D^2 \to \mathbb{P}(\mathcal{L})$ as a discrete surface.
Implicitly understood in the term ``surface'' here is that $F$ is nondegenerate in all ways needed -- in particular, adjacent values of $F$ are not orthogonal.  

Viewing discrete isothermic surfaces as Bianchi quadrilaterals of Darboux transformations of smooth isothermic surfaces, there are numerous equivalent conditions for defining the isothermicity of $F$; here, we will take the gauge-theoretic approach, borrowing the notions of discrete vector bundle theory from \cite{burstall_NotesTransformationsIntegrable_2017, burstall_DiscreteOmeganetsGuichard_2023}, and require that for some real-valued non-vanishing edge-labeling $m_{ij}$ defined on the set of unoriented edges (so that $m_{ij} = m_{\ell k}$ and $m_{i\ell} = m_{jk}$ on any elementary quadrilateral $(ijk\ell)$), the discrete connections defined on the trivial bundle $D^2 \times \mathbb{R}^{p+1,q+1}$ given by
	\[
		\Gamma(t)_{ji} := \Gamma_{F_i}^{F_j}(1-\tfrac{t}{m_{ij}})
	\]
are flat for all $t \in \mathbb{R}$, that is,
	\[
		\Gamma(t)_{kj}\Gamma(t)_{ji}=\Gamma(t)_{k \ell}\Gamma(t)_{\ell i}.
	\]
(See \cite[Theorem~4.14]{burstall_IsothermicSubmanifoldsSymmetric_2011}, \cite[Lemma~2.5]{burstall_DiscreteSurfacesConstant_2014}, or \cite[Theorem 4.5.29]{cho_DiscreteIsothermicSurfaces_2025}.)
\begin{definition}[cf.\ {\cite[Def.~3.1]{burstall_DiscreteSurfacesConstant_2014}}]
Let $F: D^2 \to \mathbb{P}(\mathcal{L})$ be a discrete isothermic surface with associated family of flat connections $\Gamma_{ji}(t)$ for any edge $(ij)$.
Then a polynomial $P(t)$ in $t$ with coefficients $P^{(\iota)} : D^2 \to \mathbb{R}^{p+1,q+1}$ is called a \emph{(discrete) polynomial conserved quantity} of $F$ if 
	\begin{equation}\label{eqn:edge}
		\Gamma(t)_{ji} P(t)_i = P(t)_j
	\end{equation}
on any edge $(ij)$.
We will refer to the condition \eqref{eqn:edge} as the \emph{degree $d$ edge property}, with the understanding that this naming implies both $P(t)_j$ and $P(t)_i$ are polynomial of degree $d$.  
\end{definition}
\begin{definition}[{\cite[Def.~3.12]{burstall_DiscreteSurfacesConstant_2014}}]
A discrete isothermic surface admitting a polynomial conserved quantity of degree $d$ is called a \emph{discrete special isothermic surface of type $d$}.
\end{definition}

\subsection{Main result}

A key observation for us is that \eqref{eq:pphat} is exactly the degree $d$ edge property in the case of a smooth surface and its B\"acklund transform (so that both $p(t)$ and $\hat p(t)$ are polynomial of degree $d$), which we discuss in the next lemma. 

\begin{lemma}\label{onlylem}
	A Darboux transform $\hat f$ of a (smooth) special isothermic surface of type $d$ is a B\"acklund transform if and only if it satisfies the degree $d$ edge property between corresponding points of $f$ and $\hat f$.  
\end{lemma}

\begin{proof}
	Suppose $\hat f$ is a B\"acklund transform of $f$ with spectral parameter $\mu$, and choose some $x \in \Sigma^2$, so that \eqref{eq:pphat} holds.
	Let us define
		\begin{gather*}
			F_i := f(x), \quad F_j := \hat{f}(x), \quad m_{ij} := \mu\\
			P(t)_i := p(t)\big|_x, \quad P(t)_j := \hat{p}(t)\big|_x.
		\end{gather*}
	Evaluating \eqref{eq:pphat} at $x$ now implies
		\[
			P(t)_j = \Gamma_{F_i}^{F_j}(1-\tfrac{t}{m_{ij}}) P(t)_i = \Gamma_{ji} P(t)_i,
		\]
	giving us the degree $d$ edge property between $f(x)$ and $\hat f(x)$.

	Conversely, if the degree $d$ edge property holds, so \eqref{eq:pphat} holds with $\hat p(t)$ polynomial (and so $p(\mu) \perp \hat f$), then this is exactly the condition for having a B\"acklund transform, as noted in Section \ref{sec:smooth}.   \qed
\end{proof}

\begin{theorem}\label{mainthm}
	For a smooth special isothermic surface of type $d$, images of a single point in the surface through a lattice of B\"acklund transforms will be a discrete special isothermic surface of type $d$.
\end{theorem}

\begin{proof}
	Let $f_i$ denote the original given special isothermic surface of type $d$, and let $f_j, f_\ell$ be the two B\"acklund transforms of $f_i$ with spectral parameter $m_{ij}$ and $m_{i\ell}$, respectively.
	Let the fourth surface $f_k$ be given via the permutability of Darboux transformations for isothermic surfaces, so that $f_k$ is simultaneously a Darboux transform of $f_j$ and $f_\ell$, with spectral parameters $m_{i\ell}$ and $m_{ij}$.
	Defining a discrete map via $F_* := f_*(x)$ at any fixed point $x\in \Sigma^2$, we know that $F$ satisfies the condition for being discrete isothermic on the elementary quadrilateral $(ijk\ell)$ by \cite[Lemma~4.7]{burstall_IsothermicSubmanifoldsSymmetric_2011} (see also \cite[Lemma 3.6.49]{cho_DiscreteIsothermicSurfaces_2025}), and Lemma~\ref{onlylem} implies that $F$ further satisfies the degree $d$ edge property on edges $(ij)$ and $(i\ell)$.
	
	Now, as proven in \cite[Theorem~3.6]{burstall_SpecialIsothermicSurfaces_2012}, $f_k$ must also be a special isothermic surface of type $d$, so that $f_k$ is a simultaneous B\"acklund transform of $f_j$ and $f_\ell$.
	Thus, $F$ also satisfies the degree $d$ edge property on edges $(jk)$ and $(\ell k)$, and $F$ is a discrete special isothermic surface of type $d$. \qed
\end{proof}

\section{Applications to various surface classes}\label{sec:three}

Theorem~\ref{mainthm} implies that the discretization principle using Bianchi quadrilaterals of permutability applies to any smooth and discrete surface classes admitting a characterization via polynomial conserved quantities.
Here, we apply this principle to some well-known examples.

\subsection{Integrable reductions for isothermic surfaces to constant mean curvature surfaces}
Let us first focus on the case $p = 3$ and $q = 0$, so that we are interested in isothermic surfaces in the conformal $3$-sphere.
Smooth special isothermic surfaces of type $1$ admitting linear conserved quantities $p(t) = \mathfrak{q} + t Y$ with normalization $\langle Y,Y\rangle = 1$ have a characterization as constant mean curvature (cmc) $H = -\langle \mathfrak{q}, Y\rangle$ surfaces in a $3$-dimensional Riemannian space form $M$ determined by the space form vector $\mathfrak{q}$ \cite[Proposition~2.5]{burstall_SpecialIsothermicSurfaces_2012} (see also \cite[Theorem~3.8.66]{cho_DiscreteIsothermicSurfaces_2025}).
Then a B\"acklund transform $\hat{f}$ of $f$ is also a cmc $H$ surface in the same space form $M$ with the same constant mean curvature, which we refer to as a CMC Darboux transform.
This transformation is equivalent to the classical Bianchi-B\"acklund transformation for cmc surfaces \cite{choSimpleFactorDressings2019, hertrich-jerominRemarksDarbouxTransform1997, kobayashiCharacterizationsBianchiBacklundTransformations2005} when the space form is Euclidean space $\mathbb{R}^3$.

Discrete cmc surfaces in space forms also admit a characterization as discrete special isothermic surfaces of type $1$ \cite[Theorem~5.5]{burstall_DiscreteSurfacesConstant_2014} (see also \cite[Theorem~4.10.67]{cho_DiscreteIsothermicSurfaces_2025}).
Therefore, we conclude:
\begin{corollary}
	Discrete constant mean curvature $H$ surfaces in $3$-dimensional Riemannian space forms of constant sectional curvature are obtained as Bianchi quadrilaterals of CMC Darboux transformations.
	Discrete constant mean curvature $H$ surfaces in Euclidean $3$-space are obtained as Bianchi quadrilaterals of Bianchi-B\"acklund transformations.
\end{corollary}

\subsection{$\Omega$-surfaces and their integrable reductions}
Now let $p = 3$ and $q = 1$ so that the projectivized light cone $\mathbb{P}(\mathcal{L})$ is the \emph{Lie quadric}, the space of oriented spheres.
Then $Z := \{ \text{lines in the Lie quadric}\}$ represents the set of contact elements, and becomes a contact manifold \cite{cecilLieSphereGeometry2008}.
Legendre maps are given as maps $\Lambda : \Sigma^2 \to Z$ viewed as rank $2$ null subbundles of the trivial bundle $\Sigma^2 \times \mathbb{R}^{4,2}$, and the surface $\mathfrak{f}$ in the space form $M$ can be recovered by taking sections $\mathfrak{f}$ of $\Lambda$ so that $\mathfrak{f}$ takes values in $M \cap \mathfrak{p}^\perp$ for some choice of the \emph{point sphere complex} $\mathfrak{p} \in \mathbb{R}^{4,2}$ with $\mathfrak{p} \perp \mathfrak{q}$ and $\mathfrak{p} \neq \mathfrak{q}$.

An $\Omega$-surface $\Lambda : \Sigma^2 \to Z$ is a Legendre immersion such that $\Lambda$ is spanned by an isotropic Darboux pair $f^\pm : \Sigma^2 \to \mathbb{P}(\mathcal{L})$ \cite{demoulinSurfacesOmega1911, demoulinSurfacesRSurfaces1911, demoulinSurfacesRSurfaces1911a, pember_LieApplicableSurfaces_2020}.
Therefore an $\Omega$-surface is a surface enveloped by \emph{a pair of isothermic sphere congruences}.
The \emph{Lie-Darboux transformation}  \cite{pember_LieApplicableSurfaces_2020}  with spectral parameter $\mu$ of an $\Omega$-surface $\Lambda = f^+ \oplus f^-$ is then induced from Darboux transformations of isothermic sphere congruences, given by the permutability between isotropic Darboux transformation and Darboux transformation with finite spectral parameter (Theorem~\ref{thm:permiso}).

A discrete $\Omega$-surface is characterized as the span of an isotropic Darboux pair of (discrete) isothermic sphere congruences in \cite[Theorem~6.2]{burstall_DiscreteOmeganetsGuichard_2023}, allowing us to deduce the following.
\begin{theorem}
	Discrete $\Omega$-surfaces are obtained as Bianchi quadrilaterals of Lie-Darboux transformations.
\end{theorem}

As the transformation theory of $\Omega$-surfaces is induced by that of the isothermic sphere congruences, the integrable reductions via polynomial conserved quantities naturally apply to $\Omega$-surfaces.
\begin{corollary}
	Discrete $\Omega$-surfaces enveloped by discrete special isothermic sphere congruences of type $d$ are obtained as Bianchi quadrilaterals of the B\"acklund-type Lie-Darboux transformations.
\end{corollary}

\subsubsection{Isothermic surfaces}
Isothermic surfaces in Riemannian space forms are $\Omega$-surfaces such that one of the isothermic sphere congruences admits a timelike constant conserved quantity, where the constant conserved quantity serves as the point sphere complex.
An integrable reduction via timelike constant conserved quantity from the case of $\Omega$-surfaces to isothermic surfaces, we recover yet another approach to showing that discrete isothermic surfaces are obtained as Bianchi quadrilaterals of Darboux transformations.

When one of the isothermic sphere congruences admits a spacelike constant conserved quantity, then the $\Omega$-surface projects to an isothermic surface in Lorentzian space forms with constant sectional curvatures: Minkowski $3$-space, de Sitter $3$-space and anti-de Sitter $3$-space.
Thus, we also obtain an analogous characterization of discrete isothermic surfaces in these space forms via permutability (see, for example, \cite[Definition~4.1]{yasumotoDiscreteMaximalSurfaces2015}).

\subsubsection{Guichard surfaces}
Guichard surfaces \cite{guichardSurfacesIsothermiques1900} constitute another integrable class of surfaces, that is a subclass of $\Omega$-surfaces, and include well-known surface classes, including pseudospherical surfaces.
These surfaces are characterized within the class of $\Omega$-surfaces by one of the isothermic sphere congruences admitting a linear conserved quantity $p(t)$ such that $(p(t),p(t))$ is linear in $t$ with non-zero constant term \cite[Theorem~5.7]{burstall_PolynomialConservedQuantities_2019}.
When one considers B\"{a}cklund-type Lie-Darboux transformations for Guichard surfaces, then it is shown in \cite[Section~5.2.2]{burstall_PolynomialConservedQuantities_2019} that these are equivalent to Eisenhart transformations \cite{eisenhart_TransformationsSurfacesGuichard_1914}.

On the other hand, discrete Guichard surfaces \cite[p.\ 383]{schiefUnificationClassicalNovel2003a} are also given a similar characterization using linear conserved quantities of the enveloping discrete isothermic sphere congruence in \cite[Theorem~7.11]{burstall_DiscreteOmeganetsGuichard_2023}.
Therefore, we conclude as follows.
\begin{corollary}
	Discrete Guichard surfaces are obtained as Bianchi quadrilaterals of Eisenhart transformations.
\end{corollary}

\subsubsection{$L$-isothermic surfaces}

$L$-isothermic surfaces are an integrable class of surfaces characterized by admitting curvature line coordinates for which the third fundamental form is conformal.
They are another subclass of $\Omega$-surfaces, where one of the isothermic sphere congruences admits a constant lightlike conserved quantity. 
The B\"{a}cklund-type Lie-Darboux transformations of $L$-isothermic surfaces coincide with Bianchi-Darboux transformations of $L$-isothermic surfaces \cite{musso_BianchiDarbouxTransformLisothermic_2000}.

As discrete $L$-isothermic surfaces can also be characterized by the existence of constant lightlike conserved quantities, we conclude as follows:
\begin{corollary}
	Discrete $L$-isothermic surfaces are obtained as Bianchi quadrilaterals of Bianchi-Darboux transformations.
\end{corollary}

\subsubsection{Linear Weingarten surfaces}
Non-tubular linear Weingarten surfaces in space forms are those surfaces that satisfy an affine linear relation between the Gauss and mean curvatures.
They admit a characterization as a subclass of $\Omega$-surfaces where the pair of isothermic sphere congruences both admit constant conserved quantities \cite{burstallLieGeometryLinear2012}.
Discrete linear Weingarten surfaces that are also circular nets admit an analogous characterization in terms of constant conserved quantities \cite{burstall_DiscreteLinearWeingarten_2018}.
\begin{corollary}
	Discrete linear Weingarten surfaces are obtained as Bianchi quadrilaterals of Lie-Darboux transformations keeping the linear Weingarten condition.
\end{corollary}

Further integrable reductions can be considered: a \emph{linear Weingarten surface of Bryant-type} or a  \emph{linear Weingarten surface of Bianchi-type} admits a characterization where the pair of isothermic sphere congruences have constant conserved quantities, and one of them is lightlike \cite{burstallLieGeometryLinear2012}.
An analogous characterization exists for discrete linear Weingarten surfaces of Bryant-type  or Bianchi-type \cite{pemberDiscreteWeierstrasstypeRepresentations2023}.
\begin{corollary}
	Discrete linear Weingarten surfaces of Bryant-type or Bianchi-type are obtained as Bianchi quadrilaterals of Lie-Darboux transformations keeping the Bryant-type or Bianchi-type linear Weingarten condition, respectively.
\end{corollary}



\begin{bibdiv}
\begin{biblist}

\bib{bianchi_LezioniDiGeometria_1903}{book}{
      author={Bianchi, Luigi},
       title={{Lezioni di Geometria Differenziale, Volume II}},
     edition={Seconda edizione},
   publisher={Enrico Spoerri},
     address={Pisa},
        date={1903},
}

\bib{bianchi_ComplementiAlleRicerche_1906}{article}{
      author={Bianchi, Luigi},
       title={Complementi alle ricerche sulle superficie isoterme},
        date={1906},
     journal={Ann. Mat. Pura Appl. (3)},
      volume={12},
      number={1},
       pages={19\ndash 54},
       doi={10.1007/BF02419495}
}

\bib{bobenko_SconicalCMCSurfaces_2016}{incollection}{
      author={Bobenko, Alexander~I.},
      author={Hoffmann, Tim},
       title={S-conical {{CMC}} surfaces. {{Towards}} a unified theory of
  discrete surfaces with constant mean curvature},
        date={2016},
   booktitle={Advances in discrete differential geometry},
   publisher={Springer},
     address={Berlin},
       pages={287\ndash 308},
      review={\MR{3587190}},
      doi={10.1007/978-3-662-50447-5_9}
}

\bib{bobenko_DiscreteIsothermicSurfaces_1996}{article}{
      author={Bobenko, Alexander~I.},
      author={Pinkall, Ulrich},
       title={Discrete isothermic surfaces},
        date={1996},
     journal={J. Reine Angew. Math.},
      volume={475},
       pages={187\ndash 208},
      review={\MR{1396732}},
      doi={10.1515/crll.1996.475.187}
}

\bib{bobenko_DiscreteSurfacesConstant_1996}{article}{
      author={Bobenko, Alexander~I.},
      author={Pinkall, Ulrich},
       title={Discrete surfaces with constant negative {{Gaussian}} curvature
  and the {{Hirota}} equation},
        date={1996},
     journal={J. Differential Geom.},
      volume={43},
      number={3},
       pages={527\ndash 611},
      review={\MR{1412677}},
      doi={10.4310/jdg/1214458324}
}

\bib{bobenko_DiscretizationSurfacesIntegrable_1999}{incollection}{
      author={Bobenko, Alexander~I.},
      author={Pinkall, Ulrich},
       title={Discretization of surfaces and integrable systems},
        date={1999},
   book={
   	title={Discrete integrable geometry and physics ({{Vienna}}, 1996)},
      editor={Bobenko, Alexander~I.},
      editor={Seiler, Ruedi},
      series={Oxford {{Lecture Ser}}. {{Math}}. {{Appl}}.},
      volume={16},
   publisher={Oxford Univ. Press},
     address={New York},},
       pages={3\ndash 58},
      review={\MR{1676682}},
}

\bib{burstall_NotesTransformationsIntegrable_2017}{incollection}{
      author={Burstall, Francis~E.},
       title={Notes on transformations in integrable geometry},
        date={2017},
   book={
   	title={Special metrics and group actions in geometry},
      editor={Chiossi, Simon~G.},
      editor={Fino, Anna},
      editor={Musso, Emilio},
      editor={Podest{\`a}, Fabio},
      editor={Vezzoni, Luigi},
      series={Springer {{INdAM Ser}}.},
      volume={23},
   publisher={Springer},
     address={Cham},},
       pages={59\ndash 80},
      review={\MR{3751962}},
      doi={10.1007/978-3-319-67519-0_3}
}

\bib{burstall_ConformalSubmanifoldGeometry_2010}{article}{
      author={Burstall, Francis~E.},
      author={Calderbank, David M.~J.},
       title={Conformal submanifold geometry {{I-III}}},
        date={2010},
      eprint={arXiv:1006.5700},
      url={https://arxiv.org/abs/1006.5700}
}

\bib{burstall_DiscreteOmeganetsGuichard_2023}{article}{
      author={Burstall, Francis~E.},
      author={Cho, Joseph},
      author={Hertrich-Jeromin, Udo},
      author={Pember, Mason},
      author={Rossman, Wayne},
       title={Discrete $\Omega$-nets and Guichard nets via discrete Koenigs
  nets},
        date={2023},
     journal={Proc. Lond. Math. Soc. (3)},
      volume={126},
      number={2},
       pages={790–836},
      review={\MR{4550152}},
      doi={10.1112/plms.12499}
}

\bib{burstall_IsothermicSubmanifoldsSymmetric_2011}{article}{
      author={Burstall, Francis~E.},
      author={Donaldson, Neil~M.},
      author={Pedit, Franz},
      author={Pinkall, Ulrich},
       title={Isothermic submanifolds of symmetric $R$-spaces},
        date={2011},
     journal={J. Reine Angew. Math.},
      volume={660},
       pages={191–243},
      review={\MR{2855825}},
      doi={10.1515/crelle.2011.075}
}

\bib{burstall_HarmonicToriSymmetric_1993}{article}{
      author={Burstall, Francis~E.},
      author={Ferus, Dirk},
      author={Pedit, Franz},
      author={Pinkall, Ulrich},
       title={Harmonic tori in symmetric spaces and commuting {{Hamiltonian}}
  systems on loop algebras},
        date={1993},
     journal={Ann. of Math. (2)},
      volume={138},
      number={1},
       pages={173\ndash 212},
      review={\MR{1230929}},
      doi={10.2307/2946637}
}

\bib{burstall_PolynomialConservedQuantities_2019}{article}{
      author={Burstall, Francis~E.},
      author={{Hertrich-Jeromin}, Udo},
      author={Pember, Mason},
      author={Rossman, Wayne},
       title={Polynomial conserved quantities of {{Lie}} applicable surfaces},
        date={2019},
     journal={Manuscripta Math.},
      volume={158},
      number={3-4},
       pages={505\ndash 546},
      review={\MR{3914961}},
      doi={10.1007/s00229-018-1033-0}
}

\bib{burstallLieGeometryLinear2012}{article}{
      author={Burstall, Francis~E.},
      author={{Hertrich-Jeromin}, Udo},
      author={Rossman, Wayne},
       title={Lie geometry of linear {{Weingarten}} surfaces},
        date={2012},
     journal={C. R. Math. Acad. Sci. Paris},
      volume={350},
      number={7-8},
       pages={413\ndash 416},
      review={\MR{2922095}},
      doi={10.1016/j.crma.2012.03.018}
}

\bib{burstall_DiscreteLinearWeingarten_2018}{article}{
      author={Burstall, Francis~E.},
      author={{Hertrich-Jeromin}, Udo},
      author={Rossman, Wayne},
       title={Discrete linear {{Weingarten}} surfaces},
        date={2018},
     journal={Nagoya Math. J.},
      volume={231},
       pages={55\ndash 88},
      review={\MR{3845588}},
      doi={10.1017/nmj.2017.11}
}

\bib{burstall_DiscreteSurfacesConstant_2014}{incollection}{
      author={Burstall, Francis~E.},
      author={{Hertrich-Jeromin}, Udo},
      author={Rossman, Wayne},
      author={Santos, Susana~D.},
       title={Discrete surfaces of constant mean curvature},
        date={2014},
   book={
   	title={Development in differential geometry of submanifolds},
      editor={Kobayashi, Shim-Pei},
      series={{{RIMS K\^oky\^uroku}}},
      volume={1880},
   publisher={Res. Inst. Math. Sci. (RIMS)},
     address={Kyoto},},
       pages={133\ndash 179},
              eprint={arXiv:0804.2707},
       url={https://arxiv.org/abs/0804.2707}
}

\bib{burstall_DiscreteSpecialIsothermic_2015}{article}{
      author={Burstall, Francis~E.},
      author={{Hertrich-Jeromin}, Udo},
      author={Rossman, Wayne},
      author={Santos, Susana~D.},
       title={Discrete special isothermic surfaces},
        date={2015},
     journal={Geom. Dedicata},
      volume={174},
       pages={1\ndash 11},
      review={\MR{3303037}},
      doi={10.1007/s10711-014-0001-4}
}

\bib{burstall_SpecialIsothermicSurfaces_2012}{article}{
      author={Burstall, Francis~E.},
      author={Santos, Susana~D.},
       title={Special isothermic surfaces of type $d$},
        date={2012},
     journal={J. Lond. Math. Soc. (2)},
      volume={85},
      number={2},
       pages={571–591},
      review={\MR{2901079}},
      doi={10.1112/jlms/jdr050}
}

\bib{cecilLieSphereGeometry2008}{book}{
      author={Cecil, Thomas~E.},
       title={Lie sphere geometry},
     edition={Second edition},
      series={Universitext},
   publisher={Springer},
     address={New York},
        date={2008},
        ISBN={978-0-387-74655-5},
      review={\MR{2361414}},
      doi={10.1007/978-0-387-74656-2}
}

\bib{cho_DiscreteIsothermicSurfaces_2025}{book}{
      author={Cho, Joseph},
      author={Naokawa, Kosuke},
      author={Ogata, Yuta},
      author={Pember, Mason},
      author={Rossman, Wayne},
      author={Yasumoto, Masashi},
       title={Discrete isothermic surfaces in {{Lie}} sphere geometry},
      series={Lecture {{Notes}} in {{Mathematics}}},
   publisher={Springer},
     address={Cham},
        date={2025},
      volume={2375},
        ISBN={978-3-031-95591-4 978-3-031-95592-1},
        doi={10.1007/978-3-031-95592-1},
      review={\MR{4982250}},
}

\bib{choSimpleFactorDressings2019}{article}{
      author={Cho, Joseph},
      author={Ogata, Yuta},
       title={Simple factor dressings and {{Bianchi}}--{{B\"acklund}}
  transformations},
        date={2019},
     journal={Illinois J. Math.},
      volume={63},
      number={4},
       pages={619\ndash 631},
      review={\MR{4032817}},
      doi={10.1215/00192082-7988989}
}

\bib{clarkeIntegrabilitySubmanifoldGeometry2012}{thesis}{
      author={Clarke, Daniel},
       title={Integrability in submanifold geometry},
        type={Ph.D. Thesis},
        organization={University of Bath},
        date={2012},
}

\bib{demoulinSurfacesOmega1911}{article}{
      author={Demoulin, Alphonse},
       title={Sur les surfaces $\Omega$},
        date={1911},
     journal={C. R. Acad. Sci. Paris},
      volume={153},
       pages={927–929},
}

\bib{demoulinSurfacesRSurfaces1911}{article}{
      author={Demoulin, Alphonse},
       title={Sur les surfaces $R$ et les surfaces $\Omega$},
        date={1911},
     journal={C. R. Acad. Sci. Paris},
      volume={153},
       pages={705–707},
}

\bib{demoulinSurfacesRSurfaces1911a}{article}{
      author={Demoulin, Alphonse},
       title={Sur les surfaces $R$ et les surfaces $\Omega$},
        date={1911},
     journal={C. R. Acad. Sci. Paris},
      volume={153},
       pages={590–593},
}

\bib{eisenhart_TransformationsSurfacesGuichard_1914}{article}{
      author={Eisenhart, Luther~Pfahler},
       title={Transformations of surfaces of {{Guichard}} and surfaces
  applicable to quadrics},
        date={1914},
     journal={Ann. Mat. Pura Appl. (3)},
      volume={22},
      number={1},
       pages={191\ndash 247},
       doi={10.1007/BF02419557}
}

\bib{guichardSurfacesIsothermiques1900}{article}{
      author={Guichard, C.},
       title={Sur les surfaces isothermiques},
        date={1900},
     journal={C. R. Acad. Sci. Paris},
      volume={130},
       pages={159\ndash 162},
}

\bib{hertrich-jeromin_IntroductionMobiusDifferential_2003}{book}{
      author={{Hertrich-Jeromin}, Udo},
       title={Introduction to {{M\"obius}} differential geometry},
      series={London {{Mathematical Society Lecture Note Series}}},
   publisher={Cambridge University Press},
     address={Cambridge},
        date={2003},
      volume={300},
      review={\MR{2004958}},
      doi={10.1017/CBO9780511546693}
}

\bib{hertrich-jeromin_DiscreteVersionDarboux_1999}{incollection}{
      author={{Hertrich-Jeromin}, Udo},
      author={Hoffmann, Tim},
      author={Pinkall, Ulrich},
       title={A discrete version of the {{Darboux}} transform for isothermic
  surfaces},
        date={1999},
   book={title={Discrete integrable geometry and physics ({{Vienna}}, 1996)},
      editor={Bobenko, Alexander~I.},
      editor={Seiler, Ruedi},
      series={Oxford {{Lecture Ser}}. {{Math}}. {{Appl}}.},
      volume={16},
   publisher={Oxford Univ. Press},
     address={New York},},
       pages={59\ndash 81},
      review={\MR{1676683}},
}

\bib{hertrich-jerominRemarksDarbouxTransform1997}{article}{
      author={{Hertrich-Jeromin}, Udo},
      author={Pedit, Franz},
       title={Remarks on the {{Darboux}} transform of isothermic surfaces},
        date={1997},
     journal={Doc. Math.},
      volume={2},
       pages={313\ndash 333},
      review={\MR{1487467}},
      doi={10.4171/DM/32}
}

\bib{hoffmann_DiscreteParametrizedSurface_2017}{article}{
      author={Hoffmann, Tim},
      author={Sageman-Furnas, Andrew~O.},
      author={Wardetzky, Max},
       title={A discrete parametrized surface theory in $\mathbb{R}^3$},
        date={2017},
     journal={Int. Math. Res. Not. IMRN},
      volume={2017},
      number={14},
       pages={4217–4258},
      review={\MR{3674170}},
      doi={10.1093/imrn/rnw015}
}

\bib{kobayashiCharacterizationsBianchiBacklundTransformations2005}{article}{
      author={Kobayashi, Shim-Pei},
      author={Inoguchi, Jun-ichi},
       title={Characterizations of {{Bianchi-B\"acklund}} transformations of
  constant mean curvature surfaces},
        date={2005},
     journal={Internat. J. Math.},
      volume={16},
      number={2},
       pages={101\ndash 110},
      review={\MR{2121843}},
      doi={10.1142/S0129167X05002801}
}

\bib{levi_BacklundTransformationsNonlinear_1980}{article}{
      author={Levi, D.},
      author={Benguria, R.},
       title={B\"acklund transformations and nonlinear differential difference
  equations},
        date={1980},
     journal={Proc. Nat. Acad. Sci. U.S.A.},
      volume={77},
      number={9, part 1},
       pages={5025\ndash 5027},
      review={\MR{587276}},
      doi={10.1073/pnas.77.9.5025}
}

\bib{musso_BianchiDarbouxTransformLisothermic_2000}{article}{
      author={Musso, Emilio},
      author={Nicolodi, Lorenzo},
       title={The Bianchi-Darboux transform of $L$-isothermic surfaces},
        date={2000},
     journal={Internat. J. Math.},
      volume={11},
      number={7},
       pages={911–924},
      review={\MR{1792958}},
      doi={10.1142/S0129167X00000465}
}

\bib{nijhoff_DirectLinearizationNonlinear_1983}{article}{
      author={Nijhoff, Frank~W.},
      author={Quispel, G. R.~W.},
      author={Capel, Hans~W.},
       title={Direct linearization of nonlinear difference-difference
  equations},
        date={1983},
     journal={Phys. Lett. A},
      volume={97},
      number={4},
       pages={125\ndash 128},
      review={\MR{719638}},
      doi={10.1016/0375-9601(83)90192-5}
}

\bib{pember_LieApplicableSurfaces_2020}{article}{
      author={Pember, Mason},
       title={Lie applicable surfaces},
        date={2020},
     journal={Comm. Anal. Geom.},
      volume={28},
      number={6},
       pages={1407\ndash 1450},
      review={\MR{4184823}},
      doi={10.4310/CAG.2020.v28.n6.a5}
}

\bib{pemberDiscreteWeierstrasstypeRepresentations2023}{article}{
      author={Pember, Mason},
      author={Polly, Denis},
      author={Yasumoto, Masashi},
       title={Discrete {{Weierstrass-type}} representations},
        date={2023},
     journal={Discrete Comput. Geom.},
      volume={70},
      number={3},
       pages={816\ndash 844},
      review={\MR{4650025}},
      doi={10.1007/s00454-022-00439-z}
}

\bib{schiefUnificationClassicalNovel2003a}{article}{
      author={Schief, Wolfgang~K.},
       title={On the unification of classical and novel integrable surfaces.
  {{II}}. {{Difference}} geometry},
        date={2003},
     journal={R. Soc. Lond. Proc. Ser. A Math. Phys. Eng. Sci.},
      volume={459},
      number={2030},
       pages={373\ndash 391},
      review={\MR{1997461}},
      doi={10.1098/rspa.2002.1008}
}

\bib{yasumotoDiscreteMaximalSurfaces2015}{article}{
      author={Yasumoto, Masashi},
       title={Discrete maximal surfaces with singularities in {{Minkowski}}
  space},
        date={2015},
     journal={Differential Geom. Appl.},
      volume={43},
       pages={130\ndash 154},
      review={\MR{3421881}},
      doi={10.1016/j.difgeo.2015.09.006}
}

\end{biblist}
\end{bibdiv}
\end{document}